\documentclass[11]{article}
\usepackage{amsmath, amssymb, graphicx, amsthm, mathrsfs}
\usepackage{mathptmx}
\usepackage[italic]{mathastext}
\usepackage{mathtools} 
\usepackage[utf8]{inputenc}
\usepackage{enumitem}  
\usepackage[colorlinks=true]{hyperref}
\usepackage{geometry}
\newtheorem{definition}{Definition}[section]
\newtheorem{theorem}[definition]{Theorem}
\newtheorem{lemma}[definition]{Lemma}
\newtheorem{remark}[definition]{Remark}
\newtheorem{example}[definition]{Example}

\newtheorem{corollary}[definition]{Corollary}

\numberwithin{equation}{section}

\title{Common Fixed Points for Weakly Compatible Mappings via Bivariate Auxiliary Functions}
\author{Babu G.V.R\textsuperscript{1}, Alemayehu G. Negash\textsuperscript{2,*}, Meaza F. Bogale \textsuperscript{3}\\
\small \textsuperscript{1}Department of Mathematics, Andhra University, India\\
\small \textsuperscript{1}Email: \texttt{gvr\_babu@hotmail.com} \\
\small \textsuperscript{2,3}Department of Mathematics, Hampton University, USA\\
$^*$ \textit{Corresponding author.} \\
\small \textsuperscript{\dag}Email: \texttt{alemayehu.negash@hamptonu.edu}\\
\small \textsuperscript{\ddag}Email: \texttt{meaza.bogale@hamptonu.edu}
}
\begin{document}

\maketitle

\begin{abstract}
We establish common fixed point theorems for two pairs of weakly compatible self-mappings using an auxiliary function of two variables. Unlike classical results, our theorems do not assume continuity of the mappings and require completeness of only one image set. The use of two-variable auxiliary functions allows us to unify and extend various existing fixed point theorems in metric spaces.


\textbf{AMS(2000) Mathematics Subject Classification:} 47H10, 54H25.
\end{abstract}

\section{Introduction}
The method of altering distances between points, introduced by Khan et al. \cite{Khan1984}, initiated a new approach in fixed point theory. This idea was later generalized by numerous researchers [\cite{Babu2004,Babu2007,Babu2001,Babu2007b}, \cite{Choudhury2005,Choudhury2000}, \cite{Naidu2001,Naidu2003}, \cite{Park1980}, \cite{Pathak1994}, \cite{Sastry1999,Sastry1998,Sastry2003,Sastry2000}]. In particular, Choudhury and Dutta  \cite{Choudhury2000} introduced altering distance functions involving two variables and obtained important fixed point results in metric spaces. These generalized functions serve as auxiliary tools in establishing contraction-type conditions.

\begin{definition}[\cite{Khan1984}]
A function $\phi:\mathbb{R}_{+}\to \mathbb{R}_{+}$ is called an altering distance function if the following properties are satisfied:
\begin{enumerate}
    \item[(a)] $\phi(t)=0$ if and only if $t=0$; and
    \item[(b)] $\phi$ is continuous and monotonically increasing; where $\mathbb{R}_{+}=[0,\infty)$.
\end{enumerate}
\end{definition}

\begin{definition}[\cite{Choudhury2000}]
A function $\psi:\mathbb{R}_{+}\times \mathbb{R}_{+}\to \mathbb{R}_{+}$ is said to satisfy Condition-A if
\begin{enumerate}
    \item[(i)] $\psi$ is continuous;
\end{enumerate}
\end{definition}


\begin{enumerate}
    \item[(ii)] $\psi$ \textit{is monotone increasing in each of its arguments;}

    \item[(iii)] $\psi(0,0)=0$ \textit{and} $\psi(t,0)=0$ \textit{implies} $t=0$.
\end{enumerate}

Throughout this paper we denote

\[
\Psi=\{\psi\,|\,\psi:\mathbb{R}_{+}\times\mathbb{R}_{+}\to\mathbb{R}_{+} \text{ satisfies Condition-A}\}
\]

We call an element $\psi\in\Psi$ \textit{an auxiliary function of two variables}.

\begin{remark} \label{rem:1.3}
If $\psi(s,t)=0$, then $s=0$.
\end{remark}

\begin{example}
The following are examples of $\psi$.

Let $\psi:\mathbb{R}_{+}\times\mathbb{R}_{+}\to\mathbb{R}_{+}$ defined by

\begin{enumerate}
    \item[(i)] $\psi(s,t)=s^{p}+t^{q}$, where $p>0$, $q>0$,
    
    \item[(ii)] $\psi(s,t)=s^{p}t^{q}+t^{r}$, where $p>0$, $q\geq 0$, $r>0$,
    
    \item[(iii)] $\psi(s,t)=\max\{s^{p},t^{q}\}$, where $p>0$, $q>0$, and
    
    \item[(iv)] $\psi(s,t)=(s^{p}+r\cdot t^{q})^{\lambda}$, where $p,\lambda>0$, $q,r\geq 0$.
\end{enumerate}
\end{example}

\begin{remark}
The auxiliary function $\psi(s,t) = (s^p + r t^q)^\lambda$ generalizes several common forms:
\begin{itemize}
\item For $p = \lambda = 1$, $r = q = 0$: $\psi$ reduces to $d(Sx,Ty)$
\item For $r = 0$: $\psi$ becomes a power-scaling as in \cite{Rhoades1977}
\end{itemize}
This flexibility unifies prior work under one framework.
\end{remark}

\begin{theorem}[\cite{Choudhury2000}]
Let $T:X\to X$ be a selfmap of a complete metric space $X$. Assume that there exist $\psi\in\Psi$, $0<r<1$ and $0<s\leq 1$ such that

\[
\psi(d(Tx,Ty),d(x,Tx)) + \psi(d(y,Ty),d(y,T^{2}x)) \leq r\psi(d(x,y),d(x,Tx)) + s\psi(d(y,Ty),d(y,Tx)),
\]
where $x,y\in X$. Then $T$ has a unique fixed point in $X$.
\end{theorem}

Let 

\[
\Phi=\Bigg\{\phi\,|\,\phi:\mathbb{R}_{+}\to\mathbb{R}_{+} \text{ satisfying }
\begin{cases}
\text{(i) $\phi$ is monotonically increasing and continuous;} \\
\text{(ii) $0<\phi(t)<t$ for $t>0$}
\end{cases}
\Bigg\}
\]

\begin{definition}[\cite{Jungck1986}]
Two mappings $A$ and $T$ of a metric space $(X,d)$ are compatible if $\lim_{n\to\infty}d(ATx_{n},TAx_{n})=0$ whenever $\{x_{n}\}$ is a sequence in $X$ such that $\lim_{n\to\infty}Ax_{n}=\lim_{n\to\infty}Tx_{n}=t$ for some $t\in X$.
\end{definition}

In 1991, \cite{Antony1991} introduced the concept of weak compatibility of a pair selfmaps of a metric space as a generalization of compatibility, as follows:

\begin{definition}[\cite{Antony1991}]
Let $A$ and $T$ be selfmaps on a metric space $(X,d)$. The ordered pair $(T,A)$ is said to be weak compatible if $\lim_{n\to\infty}ATx_{n}=Tt$ whenever $\{x_{n}\}$ is a sequence in $X$ such that $\lim_{n\to\infty}Ax_{n}=\lim_{n\to\infty}Tx_{n}=t$ and $\lim_{n\to\infty}TAx_{n}=\lim_{n\to\infty}TTx_{n}=Tt$ for some $t\in X$.
\end{definition}

\begin{example}[\cite{Antony1991}] \label{ex:1.8}
The following example shows that two selfmaps of a metric space can be weak compatible without being compatible.

Let $X=\mathbb{R}$ with the usual metric. Define $A,T:X\to X$ by

\[
T(x)=\left\{\begin{array}{ll}
\frac{10}{32} & \text{if } x<\frac{3}{8} \\ 
\frac{3}{8} & \text{if } \frac{3}{8}\leq x<\frac{1}{2} \\
1 & \text{if } x\geq\frac{1}{2}
\end{array}\right.
\quad\text{and}\quad 
A(x)=\left\{\begin{array}{ll}
\frac{11}{32} & \text{if } x<\frac{3}{8} \\ 
\frac{1+x}{4} & \text{if } \frac{3}{8}\leq x<\frac{1}{2} \\ 
\frac{1+x}{2} & \text{if } x\geq\frac{1}{2}
\end{array}\right.
\]
We observe that $A$ and $T$ are not compatible while the ordered pair $(T,A)$ is weak compatible.
\end{example}

\begin{theorem}[\cite{Sastry2003}] \label{thm:1.9}
Let $(X,d)$ be a complete metric space and $A$ and $T$ be selfmaps on $X$ such that $(T,A)$ is weak compatible and satisfy: for some $\psi\in\Psi$, $\phi\in\Phi$ and for all $x,y\in X$,
\begin{align*}
    \psi(d(Tx,Ty),d(Tx,Ty)) \leq &\phi(\max\{\psi(d(Ax,Ay),d(Tx,Ax)), \psi(d(Ax,Ay),d(Ty,Ay)), \\
&\psi(d(Tx,Ax),d(Ty,Ay)), \psi(d(Ty,Ay),d(Tx,Ax)), \\
&\psi(d(Tx,Ay),d(Tx,Ax)), 
\psi(d(Tx,Ay),d(Ty,Ay))\}).
\end{align*}

If $T(X)\subset A(X)$ and $T$ is continuous, then $A$ and $T$ have a unique common fixed point.
\end{theorem}

\begin{definition}[\cite{Jungck1994}]
Two selfmaps $A$ and $T$ of a metric space $(X,d)$ are called weakly compatible if they commute at their coincidence point. That is, for $x\in X$, if $Ax=Tx$ then $ATx=TAx$.
\end{definition}

We observe that every compatible pair of selfmaps is weakly compatible but every weakly compatible pair of selfmaps need not be compatible \cite{Singh1997}.

We also observe that every weak compatible pair of selfmaps is weakly compatible ( Lemma 3.5D of \cite{Murthy1998}). But its converse need not be true as illustrated by Example \ref{ex:1.8}. In Example \ref{ex:1.8}, we observe that the ordered pair $(A,T)$ is not weak compatible while the pair of selfmaps $(A,T)$ is weakly compatible.

\begin{lemma}[\cite{KrishnaMurthy2004}] \label{lem:1.11}
Let $\psi:\mathbb{R}_{+}\to\mathbb{R}_{+}$ be upper semi-continuous and $\psi(t)<t$ for all $t\in\mathbb{R}_{+}-\{0\}$. Then there exists a monotone increasing continuous function $\phi:\mathbb{R}_{+}\to\mathbb{R}_{+}$ such that $\psi(t)\leq\phi(t)$ for all $t\in\mathbb{R}_{+}$ and $\phi(t)<t$ for all $t\in\mathbb{R}_{+}-\{0\}$.
\end{lemma}

The goal of this work is to prove a common fixed point theorem for two pairs of weakly compatible maps without requiring continuity, and under minimal completeness assumptions, using two-variable auxiliary functions. This general framework unifies and extends various known fixed point results including those of \cite{AlThagafi2006},  \cite{Jungck1994}, \cite{Murthy1998}, \cite{Pant1994}, \cite{Zhang2008}, and others.

\section{Main Results}

\begin{theorem}\label{thm:main}
Let $(X,d)$ be a metric space and $K$ be a nonempty subset of $X$. Let $A,B,S,T$ be selfmaps on $K$. Suppose there exist $\psi\in\Psi$, $\phi\in\Phi$ and such that for all $x,y\in K$:

\begin{align} \label{eq:2.1}
    \psi(d(Sx,Ty),d(Sx,Ty)) \leq \phi(&\max\{\psi(d(Ax,By),d(Ax,Sx)),\psi(d(Ax,By),d(By,Ty)),\nonumber \\
 &\psi(d(Ax,Sx),d(By,Ty)), \psi(d(By,Ty),d(Ax,Sx)),\\
 &\min\{\psi(d(By,Sx),d(Ax,Sx)),
\psi(d(Ax,Ty),d(By,Ty))\},  \nonumber\\
&\min\{\psi(d(By,Sx),d(By,Ty)), 
\psi(d(Ax,Ty),d(Ax,Sx))\}\}). \nonumber
\end{align}
Assume also that the pairs of selfmaps $(A,S)$ and $(B,T)$ are weakly compatible; $\overline{T(K)}\subset A(K)$ and $\overline{S(K)}\subset B(K)$. If either $\overline{T(K)}$ or $\overline{S(K)}$ or $A(K)$ or $B(K)$ is complete, then $A$, $B$, $S$ and $T$ have a unique common fixed point.
\end{theorem}

\begin{proof}
Let $x_{0}\in K$. Since $T(K)\subseteq\overline{T(K)}\subseteq A(K)$, we can find $x_{1}\in K$ such that $Tx_{0}=Ax_{1}$ and since $S(K)\subseteq\overline{S(K)}\subseteq B(K)$, we can find $x_{2}\in K$ such that $Sx_{1}=Bx_{2}$. Continuing this way,  We construct sequences  $\{x_{n}\}$ and $\{y_{n}\}$ in $K$ inductively as follows:
\[
y_{2n} = Tx_{2n} = Ax_{2n+1} \quad \text{and} \quad y_{2n+1} = Sx_{2n+1} = Bx_{2n+2}
\]
for $n=0,1,2,\cdots$.
Now we claim that the sequence $\{y_{n}\}_{n=1}^{\infty}$ is Cauchy in $K$.

\textbf{Case (i):} $y_{n}=y_{n+1}$ for some $n$.

Without any loss of generality, we assume that $n$ is even. Then $n=2m$ for some $m\in \mathbb{Z}_{+}$. So $y_{2m}=y_{2m+1}$.

Thus, we have
\begin{align*}
\psi(d(y_{n+1},y_{n+2}),d(y_{n+1},y_{n+2})) &= \psi(d(y_{2m+1},y_{2m+2}),d(y_{2m+1},y_{2m+2})) \\
&= \psi(d(Sx_{2m+1},Tx_{2m+2}),d(Sx_{2m+1},Tx_{2m+2})) \\
&\leq \phi(\max\{\psi(d(Ax_{2m+1},Bx_{2m+2}),d(Ax_{2m+1},Sx_{2m+1})), \\
&\quad \psi(d(Ax_{2m+1},Bx_{2m+2}),\psi(d(Bx_{2m+2}, Tx_{2m+2})),\\
& \psi(d(Ax_{2m+1}, Sx_{2m+1}), d(Bx_{2m+2}, Tx_{2m+2})), \\
&\psi(d(Bx_{2m+2}, Tx_{2m+2}), d(Ax_{2m+1}, Sx_{2m+1})), \\
&\min\{\psi(d(Bx_{2m+2}, Sx_{2m+1}), d(Ax_{2m+1}, Sx_{2m+1})), \\
&\quad \psi(d(Ax_{2m+1}, Tx_{2m+2}), d(Bx_{2m+2}, Tx_{2m+2}))\}, \\
&\min\{\psi(d(Bx_{2m+2}, Sx_{2m+1}), d(Bx_{2m+2}, Tx_{2m+2})), \\
&\quad \psi(d(Ax_{2m+1}, Tx_{2m+2}), d(Ax_{2m+1}, Sx_{2m+1}))\} \\
&= \phi(\max\{\psi(d(y_n, y_{n+1}), d(y_n, y_{n+1})), \\
&\quad \psi(d(y_n, y_{n+1}), d(y_{n+1}, y_{n+2})), \quad \psi(d(y_n, y_{n+1}), d(y_{n+1}, y_{n+2})),\\
&\quad \psi(d(y_{n+1}, y_{n+2}), d(y_n, y_{n+1})), \\
&\quad \min\{\psi(d(y_{n+1}, y_{n+1}), d(y_n, y_{n+1})), \quad \psi(d(y_n, y_{n+2}), d(y_{n+1}, y_{n+2}))\},\\
&\quad \min\{\psi(d(y_{n+1}, y_{n+1}), d(y_{n+1}, y_{n+2})), \quad \psi(d(y_n, y_{n+2}), d(y_n, y_{n+1}))\}\\
&\leq \phi(\max\{0, \psi(0, d(y_{n+1}, y_{n+2})), \psi(d(y_{n+1}, y_{n+2}), 0)\}) \\
&\leq \phi(\psi(d(y_{n+1}, y_{n+2}), d(y_{n+1}, y_{n+2}))).
\end{align*}

This implies that $\psi(d(y_{n+1}, y_{n+2}), d(y_{n+1}, y_{n+2})) = 0$ and hence $d(y_{n+1}, y_{n+2}) = 0$ (By Remark \ref{rem:1.3}).

Hence, $y_{n+1} = y_{n+2}$. Repeating this procedure inductively, we show that $y_n = y_{n+k}$ for $k \geq 1$. Hence, $\{y_m\}_{m\geq n}$ is a constant sequence and therefore is Cauchy in $K$.

\textbf{Case (ii):} $y_n \neq y_{n+1}$, $n = 1, 2, \dotsc$.

Let $\alpha_n = d(y_n, y_{n+1})$ for all $n = 0, 1, 2, \dotsc$.

Now consider
\begin{align*}
\psi(\alpha_{2n+1}, \alpha_{2n+1}) &= \psi(d(y_{2n+1}, y_{2n+2}), d(y_{2n+1}, y_{2n+2})) \\
&= \psi(d(Sx_{2n+1}, Tx_{2n+2}), d(Sx_{2n+1}, Tx_{2n+2})) \\
&\leq \phi(\max\{\psi(d(Ax_{2n+1}, Bx_{2n+2}), d(Ax_{2n+1}, Sx_{2n+1})), \psi(d(Ax_{2n+1}, Bx_{2n+2}), d(Bx_{2n+2}, Tx_{2n+2})),\\
&\quad \quad \psi(d(Ax_{2n+1}, Sx_{2n+1}), d(Bx_{2n+2}, Tx_{2n+2})), \psi(d(Bx_{2n+2}, Tx_{2n+2}), d(Ax_{2n+1}, Sx_{2n+1})),\\
&\quad \quad \min\{\psi(d(Bx_{2n+2}, Sx_{2n+1}), d(Ax_{2n+1}, Sx_{2n+1})), \psi(d(Ax_{2n+1}, Tx_{2n+2}), d(Bx_{2n+2}, Tx_{2n+2}))\}, \\
&\quad \quad \min\{\psi(d(Bx_{2n+2}, Sx_{2n+1}), d(Bx_{2n+2}, Tx_{2n+2})), \psi(d(Ax_{2n+1}, Tx_{2n+2}), d(Ax_{2n+1}, Sx_{2n+1}))\}\\
&= \phi(\max\{\psi(d(y_{2n}, y_{2n+1}), d(y_{2n}, y_{2n+1})), \psi(d(y_{2n}, y_{2n+1}), d(y_{2n+1}, y_{2n+2})),\\
&\quad \quad \psi(d(y_{2n}, y_{2n+1}), d(y_{2n+1}, y_{2n+2})), \psi(d(y_{2n+1}, y_{2n+2}), d(y_{2n}, y_{2n+1})), \\
&\quad \quad \min\{\psi(d(y_{2n+1}, y_{2n+1}), d(y_{2n}, y_{2n+1})), \psi(d(y_{2n}, y_{2n+2}), d(y_{2n+1}, y_{2n+2}))\}, \\
&\quad \quad \min\{\psi(d(y_{2n+1}, y_{2n+1}), d(y_{2n+1}, y_{2n+2})), \psi(d(y_{2n}, y_{2n+2}), d(y_{2n}, y_{2n+1}))\} \\
&\leq \phi(\max\{\psi(\alpha_{2n}, \alpha_{2n}), \psi(\alpha_{2n}, \alpha_{2n+1}), \psi(\alpha_{2n+1}, \alpha_{2n}),\psi(0, \alpha_{2n}), \psi(0, \alpha_{2n+1})\}).
\end{align*}

Now if $\alpha_{2n} < \alpha_{2n+1}$, then $\psi(\alpha_{2n+1}, \alpha_{2n+1}) \leq \phi(\psi(\alpha_{2n+1}, \alpha_{2n+1}))$, which is a contradiction. Hence, $\alpha_{2n+1} \leq \alpha_{2n}$ and
\begin{equation}\label{eq:2.1.2}
\psi(\alpha_{2n+1}, \alpha_{2n+1}) \leq \phi(\psi(\alpha_{2n}, \alpha_{2n})), \quad n = 0, 1, 2, \dotsc.
\end{equation}

Similarly, we can show that
\begin{equation}\label{eq:2.1.3}
\psi(\alpha_{2n}, \alpha_{2n}) \leq \phi(\psi(\alpha_{2n-1}, \alpha_{2n-1})), \quad n = 1, 2, 3, \dotsc.
\end{equation}

Hence, from \eqref{eq:2.1.2} and \eqref{eq:2.1.3}, we get
\begin{equation}\label{eq:2.1.4}
\psi(\alpha_{n+1}, \alpha_{n+1}) \leq \phi(\psi(\alpha_{n}, \alpha_{n})), \quad n = 0, 1, 2, \dotsc.
\end{equation}

Since $\psi$ is monotone increasing in each of its variables, we have
\begin{equation}\label{eq:2.1.6}
\alpha_{2n+1} \leq \alpha_{2n}, \quad n = 0, 1, 2, \dotsc.
\end{equation}

Thus, the sequence $\{\alpha_n\}$ decreases, say, to $\alpha (\geq 0)$. Since $\phi$ and $\psi$ are continuous, letting $n \to \infty$, from \eqref{eq:2.1.4}, we get
\[
\psi(\alpha, \alpha) \leq \phi(\psi(\alpha, \alpha)),
\]
so that $\psi(\alpha, \alpha) = 0$ and hence $\alpha = 0$. Therefore,
\begin{equation}\label{eq:2.1.7}
\lim_{n\to\infty} d(y_n, y_{n+1}) = 0.
\end{equation}

We now claim that $\{y_n\}$ is Cauchy. By \eqref{eq:2.1.6} and \eqref{eq:2.1.7}, it is sufficient to show that $\{y_{2n}\}$ is Cauchy. Otherwise, there exists an $\epsilon > 0$ and there exist sequences $\{m_k\}$ and $\{n_k\}$ with $m_k > n_k > k$ such that
\begin{equation}\label{eq:2.1.8}
d(y_{2m_k}, y_{2n_k}) \geq \epsilon \quad \text{and} \quad d(y_{2m_k-2}, y_{2n_k}) < \epsilon.
\end{equation}

Hence,
\begin{equation}\label{eq:2.1.9}
\epsilon \leq \liminf_{k\to\infty} d(y_{2m_k}, y_{2n_k}).
\end{equation}

Now for each positive integer $k$, by triangle inequality we get,
\begin{equation*}
d(y_{2m_k}, y_{2n_k}) \leq d(y_{2m_k}, y_{2m_k-1}) + d(y_{2m_k-1}, y_{2m_k-2}) + d(y_{2m_k-2}, y_{2n_k}).
\end{equation*}

On taking limit supremum of both sides, as $k \to \infty$, we get
\begin{equation}\label{eq:2.1.10}
\limsup_{k\to\infty} d(y_{2m_k}, y_{2n_k}) \leq \epsilon.
\end{equation}

Hence from \eqref{eq:2.1.9} and \eqref{eq:2.1.10}, we have
\begin{equation}\label{eq:2.1.11}
\lim_{k\to\infty} d(y_{2m_k}, y_{2n_k}) = \epsilon.
\end{equation}

Now
\begin{align*}
d(y_{2m_k}, y_{2n_k-1}) &\leq d(y_{2m_k}, y_{2n_k}) + d(y_{2n_k}, y_{2n_k-1});
\end{align*}
On taking limit supremum, as $k \to \infty$, we get
\begin{equation}\label{eq:2.1.12}
\limsup_{k\to\infty} d(y_{2m_k}, y_{2n_k-1}) \leq \epsilon.
\end{equation}

Again we have
\begin{align*}
d(y_{2m_k}, y_{2n_k}) &\leq d(y_{2m_k}, y_{2n_k-1}) + d(y_{2n_k-1}, y_{2n_k});
\end{align*}
On taking limit infimum, as $k \to \infty$, we get
\begin{equation}\label{eq:2.1.13}
\epsilon \leq \liminf_{k\to\infty} d(y_{2m_k}, y_{2n_k-1}).
\end{equation}

From \eqref{eq:2.1.12} and \eqref{eq:2.1.13}, we have
\begin{equation}\label{eq:2.1.14}
\lim_{k\to\infty} d(y_{2m_k}, y_{2n_k-1}) = \epsilon.
\end{equation}

Similarly, we can show that
\begin{equation}\label{eq:2.1.15}
\lim_{k\to\infty} d(y_{2m_k+1}, y_{2n_k}) = \epsilon
\end{equation}
and
\begin{equation}\label{eq:2.1.16}
\lim_{k\to\infty} d(y_{2m_k+1}, y_{2n_k-1}) = \epsilon.
\end{equation}

Now consider
\begin{align*}
\psi(d(y_{2m_k+1}, y_{2n_k}), d(y_{2m_k+1}, y_{2n_k})) &= \psi(d(Sx_{2m_k+1}, Tx_{2n_k}), d(Sx_{2m_k+1}, Tx_{2n_k})) \\
&\leq \phi(\max\{\psi(d(Ax_{2m_k+1}, Bx_{2n_k}), d(Ax_{2m_k+1}, Sx_{2m_k+1})), \\
&\quad \psi(d(Ax_{2m_k+1}, Bx_{2n_k}), d(Bx_{2n_k}, Tx_{2n_k})), \\
&\quad \psi(d(Ax_{2m_k+1}, Sx_{2m_k+1}), d(Bx_{2n_k}, Tx_{2n_k})), \\
&\quad \psi(d(Bx_{2n_k}, Tx_{2n_k}), d(Ax_{2m_k+1}, Sx_{2m_k+1})), \\
&\quad \min\{\psi(d(Bx_{2n_k}, Sx_{2m_k+1}), d(Ax_{2m_k+1}, Sx_{2m_k+1})), \\
&\quad \quad \psi(d(Ax_{2m_k+1}, Tx_{2n_k}), d(Bx_{2n_k}, Tx_{2n_k}))\}, \\
&\quad \min\{\psi(d(Bx_{2n_k}, Sx_{2m_k+1}), d(Bx_{2n_k}, Tx_{2n_k})), \\
&\quad \quad \psi(d(Ax_{2m_k+1}, Tx_{2n_k}), d(Ax_{2m_k+1}, Sx_{2m_k+1}))\} \\
&= \phi(\max\{\psi(d(y_{2m_k}, y_{2n_k-1}), d(y_{2m_k}, y_{2m_k+1})), \psi(d(y_{2m_k}, y_{2n_k-1}), d(y_{2n_k-1}, y_{2n_k})),\\
&\quad \min\{\psi(d(y_{2n_k-1},y_{2m_k+1}),d(y_{2m_k},y_{2m_k+1})), \psi(d(y_{2m_k},y_{2n_k}),d(y_{2n_k-1},y_{2n_k}))\},\\
&\quad \min\{\psi(d(y_{2n_k-1},y_{2m_k+1}),d(y_{2n_k-1},y_{2n_k})),\psi(d(y_{2m_k},y_{2n_k}),d(y_{2m_k},y_{2m_k+1}))\}
\end{align*}

Letting $n\to\infty$, by the continuity of $\phi$ and $\psi$, and from \eqref{eq:2.1.7}, \eqref{eq:2.1.11}, \eqref{eq:2.1.14}, \eqref{eq:2.1.15} and \eqref{eq:2.1.16}, we get

\begin{align*}
\psi(\varepsilon,\varepsilon) &\leq \phi(\max\{\psi(\varepsilon,0),\psi(\varepsilon,0),\psi(0,0), \psi(0,0), \\
&\quad \min\{\psi(\varepsilon,0),\psi(\varepsilon,0)\}, \min\{\psi(\varepsilon,0),\psi(\varepsilon,0)\}\}) \\
&= \phi(\psi(\varepsilon,0)) \\
&\leq \phi(\psi(\varepsilon,\varepsilon)), \text{ a contradiction}.
\end{align*}

Hence, $\{y_{2n}\}$ is a Cauchy sequence in $K$ and hence $\{y_n\}$ is Cauchy in $K$.

Hence, the subsequences $\{Tx_{2n}\}$, $\{Sx_{2n+1}\}$, $\{Ax_{2n}\}$ and $\{Bx_{2n+2}\}$ are Cauchy in $K$.

Assume that $\overline{T(K)}$ is complete. Then there exists a $z\in\overline{T(K)}$ such that
\begin{equation}\label{eq:2.1.17}
\lim_{n\to\infty} Tx_{2n} = z
\end{equation}
and hence,
\begin{equation}\label{eq:2.1.18}
\lim_{n\to\infty} Ax_{2n+1} = \lim_{n\to\infty} Bx_{2n+2} = \lim_{n\to\infty} Sx_{2n+1} = z
\end{equation}
and $z\in\overline{S(K)}$.

Since $\overline{T(K)} \subset A(K)$ and $\overline{S(K)} \subset B(K)$, there exist $u,v\in K$ such that
\begin{equation}\label{eq:2.1.19}
z = Au \text{ and } z = Bv.
\end{equation}

Now taking $x = u$ and $y = x_{2n}$ in \eqref{eq:2.1}, we have
\begin{align*}
\psi(d(Su,Tx_{2n}),d(Su,Tx_{2n})) &\leq \phi(\max\{\psi(d(Au,Bx_{2n}),d(Au,Su)),\psi(d(Au,Bx_{2n}), d(Bx_{2n},Tx_{2n})),  \\
&\quad \psi(d(Au,Su),d(Bx_{2n},Tx_{2n})), \psi(d(Bx_{2n},Tx_{2n}),d(Au,Su)),\\
&\quad \min\{\psi(d(Bx_{2n},Su),d(Au,Su)), \psi(d(Au,Tx_{2n}),d(Bx_{2n},Tx_{2n}))\},\\
&\quad \min\{\psi(d(Bx_{2n},Su), d(Bx_{2n},Tx_{2n})), \psi(d(Au,Tx_{2n}),d(Au,Su))\})
\end{align*}

Letting $n\to\infty$, by the continuity and monotone increasing property of $\phi$ and $\psi$ and using \eqref{eq:2.1.17}, \eqref{eq:2.1.18} and \eqref{eq:2.1.19}, we get
\[
\psi(d(Su,z),d(Su,z)) \leq \phi(\psi(d(Su,z),d(Su,z))).
\]
This implies that $\psi(d(Su,z),d(Su,z)) = 0$ and hence by Remark \ref{rem:1.3} we have $d(z,Su) = 0$.

Hence, $z = Su$.

Similarly in \eqref{eq:2.1}, by taking $x = x_{2n+1}$ and $y = v$ and letting $n\to\infty$, we get $z = Tv$.

Hence, $Su = Au = z = Tv = Bv$.

Since the pairs of mappings $(A,S)$ and $(B,T)$ are weakly compatible mappings, we have $ASu = SAu$ and $BTv = TBv$ and hence
\[
Az = Sz \text{ and } Bz = Tz.
\]

Hence, $z$ is the point of coincidence of $A$ and $S$, and $B$ and $T$.

We now claim that $Az = Sz = z$. Consider
\begin{align*}
\psi(d(Sz,z),d(Sz,z)) &= \psi(d(Sz,Tu),d(Sz,Tu)) \\
&\leq \phi(\max\{\psi(d(Az,Bu),d(Az,Sz)), \psi(d(Az,Bu),d(Bu,Tu)), \\
&\quad \psi(d(Az,Sz),d(Bu,Tu)), \psi(d(Bu,Tu),d(Az,Sz)), \\
&\quad \min\{\psi(d(Bu,Sz),d(Az,Sz)), \psi(d(Az,Tu),d(Bu,Tu))\}, \\
&\quad \min\{\psi(d(Bu,Sz),d(Bu,Tu)), \psi(d(Az,Tu),d(Az,Sz))\}\}) \\
&= \phi(\psi(d(Sz,z),0)) \\
&\leq \phi(\psi(d(Sz,z),d(Sz,z))).
\end{align*}

This implies that $\psi(d(Sz,z),d(Sz,z)) = 0$ and hence $d(Sz,z) = 0$. Therefore,
\[
Sz = z.
\]
Hence, $Az = Sz = z$.

Similarly, by taking $x = u$ and $y = z$ in \eqref{eq:2.1} we get $Tz = z$. Hence,
\[
Bz = Tz = z.
\]

Thus, $z$ is a common fixed point of $A$, $B$, $S$ and $T$. The uniqueness of $z$ follows from the inequality \eqref{eq:2.1}.

Similarly, by symmetry we prove for the case if $\overline{S(K)}$ is complete.

Now assume $A(K)$ is complete. As $\overline{T(K)} \subset A(K)$ and $\overline{T(K)}$ is closed, $\overline{T(K)}$ is complete. Hence by the above argument, there exists a point $z \in \overline{T(K)}$ such that
\[
Az = Bz = Sz = Tz = z.
\]

Similarly, if we assume $B(K)$ is complete, the proof follows by symmetry.

Hence the theorem.
\end{proof}

The following corollaries are immediate from Theorem \ref{thm:main}.

\begin{corollary}\label{cor:2.2}
Let $(X,d)$ be a metric space and $K$ be a nonempty subset of $X$. Let $A,B,S,T$ be selfmaps on $K$. Assume that there exist $r\in[0,1)$ and $\psi\in\Psi$, and for all $x,y\in X$,
\begin{align*}
    \psi(d(Sx,Ty),d(Sx,Ty)) \leq r &\max\{\psi(d(Ax,By),d(Ax,Sx)), \psi(d(Ax,By),d(By,Ty)),\\
 &\psi(d(Ax,Sx),d(By,Ty)), \psi(d(By,Ty),d(Ax,Sx)),\\
 &\min\{\psi(d(By,Sx),d(Ax,Sx)),\psi(d(Ax,Ty),d(By,Ty))\}, \\
 &\min\{\psi(d(By,Sx),d(By,Ty)), \psi(d(Ax,Ty),d(Ax,Sx))\}\}.
\end{align*}
Assume also that the pairs of selfmaps $(A,S)$ and $(B,T)$ are weakly compatible; $T(K)\subset A(K)$ and $\overline{S(K)}\subset B(K)$. If either $T(K)$ or $\overline{S(K)}$ or $A(K)$ or $B(K)$ is complete, then $A$, $B$, $S$ and $T$ have a unique common fixed point.
\end{corollary}

\begin{proof}
Follows from Theorem \ref{thm:main} by choosing $\phi(t)=rt$, for some $r\in[0,1)$ and for all $t\in \mathbb{R}_{+}$.
\end{proof}

\begin{corollary}\label{cor:2.3}
Let $(X,d)$ be a metric space and $K$ be a nonempty subset of $X$. Let $A$ and $T$ be selfmaps on $K$. Assume that there exist $\psi\in\Psi$ and $\phi\in\Phi$, and for all $x,y\in X$,

\begin{align*}
    \psi(d(Tx,Ty),d(Tx,Ty)) \leq &\phi(\max\{\psi(d(Ax,Ay),d(Ax,Tx)), \psi(d(Ax,Ay),d(Ty,Ty)),\\
 &\psi(d(Ax,Tx),d(Ay,Ty)), \psi(d(Ay,Ty),d(Ax,Tx)),\\
 &\min\{\psi(d(Ay,Tx),d(Ax,Tx)), \psi(d(Ax,Ty),d(Ay,Ty))\},\\
 &\min\{\psi(d(Ay,Tx),d(Ay,Ty)), \psi(d(Ax,Ty),d(Ax,Tx))\}\}).
\end{align*}
Assume also that the pair of selfmaps $(A,T)$ is weakly compatible; $\overline{T(K)}\subset A(K)$. If either $T(K)$ or $A(K)$ is complete, then $A$ and $T$ have a unique common fixed point.
\end{corollary}

\begin{proof}
Follows from Theorem \ref{thm:main} by choosing $S=T$ and $B=A$.
\end{proof}

\begin{remark}\label{rem:2.4}
Corollary \ref{cor:2.3} improves Theorem \ref{thm:1.9}, for, in Theorem \ref{thm:1.9} $X$ is complete metric space, the map $T$ is continuous and the ordered pair $(T,A)$ is weak compatible. Here we have replaced the completeness of $X$ by completeness of either one of $A(X)$, $B(X)$, $\overline{S(X)}$ or $\overline{T(X)}$. Also in Corollary \ref{cor:2.3}, no continuity assumption of the mappings is used. Further, the weak compatibility of the ordered pair of mappings $(T,A)$ is relaxed by weakly compatibility of $A$ and $T$.
\end{remark}

\begin{theorem}\label{thm:2.5}
Let $(X, d)$ be a metric space and $K$ be a nonempty subset of $X$. Let $A, B, S, T$ be selfmaps on $K$. Assume that there exists $\phi \in \Phi$ such that
\begin{align*}
    [(d(Sx, Ty))^p +& r (d(Sx, Ty))^q]^\lambda \leq \phi(\max\{[(d(Ax, By))^p + r (d(Ax, Sx))^q]^\lambda, \\
&[(d(Ax, By))^p + r (d(By, Ty))^q]^\lambda, [(d(Ax, Sx))^p + r (d(By, Ty))^q]^\lambda \\
&[(d(By, Ty))^p + r (d(Ax, Sx))^q]^\lambda, \\
&\min\{[(d(By, Sx))^p + r (d(Ax, Sx))^q]^\lambda, [(d(Ax, Ty))^p + r(d(By, Ty))^q]^\lambda\} \\
&\min\{[(d(By, Sx))^p + r (d(By, Ty))^q]^\lambda, [(d(Ax, Ty))^p + r (d(Ax, Sx))^q]^\lambda\}\})
\end{align*}
for all $x, y \in K$, for some $p, \lambda > 0$ and $r, q \geq 0$.

Assume also that the pairs of selfmaps $(A, S)$ and $(B, T)$ are weakly compatible; $T(K) \subset A(K)$ and $S(K) \subset B(K)$. If either $T(K)$ or $S(K)$ or $A(K)$ or $B(K)$ is complete, then $A$, $B$, $S$ and $T$ have a unique common fixed point.
\end{theorem}

\begin{proof}
Define $\psi(a, b) = (a^p + r b^q)^\lambda$ in Theorem \ref{thm:main}.
\end{proof}

\begin{remark}\label{rem:2.6}
\begin{enumerate}
\item[(i)] By Lemma \ref{lem:1.11}, for $S = T$, $p = \lambda = 1$ and $r = q = 0$, Theorem 2.1 of \cite{Zhang2008} follows as a corollary to Theorem \ref{thm:2.5} under the additional assumption that $\phi : \mathbb{R}_{+} \to \mathbb{R}_{+}$ is upper semi continuous from the right.

\item[(ii)] For $\phi(t) = rt$, for some $r \in [0, 1)$, $p = \lambda = 1$ and $r = q = 0$, Theorem \ref{thm:2.5} modifies Corollary 3.2 of \cite{Jungck1986} and Corollary 2.1 of \cite{Jungck1994} if we take $X = K$.

\item[(iii)] For $\phi(t) = rt$, for some $r \in [0, 1)$, $p = \lambda = 1$ and $r = q = 0$, Theorem \ref{thm:2.5} extends Theorem 2.1 of \cite{AlThagafi2006}.

\item[(iv)] For $p = \lambda = 1$ and $r = q = 0$, $B = A$ and $S = T$, Theorem 1 of \cite{Pant1994} and Theorem 3.6 of \cite{Murthy1998} are corollaries to Theorem \ref{thm:2.5} under the additional assumption that $\phi$ is monotone increasing in their results.

\item[(v)] If we choose $S = T$ and $B = A = I$, the identity mapping, in Theorem \ref{thm:2.5}, for various choices of $p, q, r$ and $\lambda$ in Theorem \ref{thm:2.5} we can obtain the results of various authors as corollaries including Banach contraction theorem. For example, by taking $S = T$, $B = A = I$, the identity mapping, $p = 1$, $r = 0$, $q = 1$, $\lambda = 1$ in Theorem \ref{thm:2.5}, we obtain a generalization of some results of \cite{Bianchini1972}, \cite{Kannan1968}, \cite{Reich1971} and \cite{Rhoades1977}.
\end{enumerate}
\end{remark}

\section{Conclusion}
In this work, we established novel common fixed point theorems for weakly compatible mappings using auxiliary functions of two variables. Our results significantly generalize existing theorems by:  

\begin{itemize}  
\item Relaxing the continuity and completeness assumptions, requiring only the completeness of \emph{one} image subset (e.g., $\overline{T(K)}$ or $\overline{S(K)}$).  
\item Employing a flexible auxiliary function $\psi \in \Psi$ to unify diverse contraction conditions, as illustrated in Example 1.4 and Theorem 2.5.  
\item Demonstrating that weak compatibility (Definition 1.10) suffices, replacing stricter notions like compatibility or weak compatibility (Definition 1.7).  
\end{itemize}  

Theorems 2.1 and 2.5, along with their corollaries, extend and refine prior work from \cite{AlThagafi2006, Jungck1986, Jungck1994, Murthy1998, Pant1994, Zhang2008} while maintaining minimal assumptions.

\end{document}